\documentclass{article}
\usepackage{graphicx}
\usepackage{amsmath}
\usepackage{amssymb}
\usepackage{amscd}
\usepackage{mathrsfs}

\newtheorem{theorem}{Theorem}

\newtheorem{theoremc}{Theorem}

\newtheorem{rk}[theoremc]{Remark}

\newtheorem{prop}[theorem]{Proposition}
\newenvironment{proof}[1][Proof]{\textbf{#1.} }{\qed \vspace{5pt}}
\newcommand\bib[1]{\bibitem[#1]{#1}}

\newcommand\qed{\phantom{\underline{y}}\hfill\hfill$\square$}
\newcommand{\comm}[1]{}
\newcommand\Cc{\let\mathcal\mathscr\mathcal C}

\renewcommand\a{\alpha}

\newcommand\C{{\mathbb C}}

\newcommand\D{{\mathcal D}}

\newcommand\E{{\mathcal E}}

\renewcommand\l{\lambda}

\renewcommand\O{\Omega}
\newcommand\oo{\omega}
\newcommand\op[1]{\mathop{\rm #1}\nolimits}
\newcommand\ot{\otimes}
\newcommand\p{\partial}

\newcommand\R{{\mathbb R}}
\renewcommand\t{\tau}

\newcommand\ti{\tilde}

\newcommand\vp{\varphi}
\newcommand\we{\wedge}
\newcommand\x{\xi}

\newcommand\z{\sigma}
\newcommand\Z{{\mathbb Z}}

\begin{document}

 \title{Point classification of 2nd order ODEs: \\
 Tresse classification revisited and beyond}
 \author{Boris Kruglikov}
 \date{}
 \maketitle

 \begin{abstract}
In 1896 Tresse gave a complete description of relative differential
invariants for the pseudogroup action of point transformations on
the 2nd order ODEs. The purpose of this paper is to review, in light
of modern geometric approach to PDEs, this classification and also
discuss the role of absolute invariants and
the equivalence problem.%
 \end{abstract}

\section*{Introduction}

Second order scalar ordinary differential equations have been the
classical target of investigations and source of inspiration for
complicated physical models. Under contact transformations all these
equations are locally equivalent, but to find such a transformation
for a pair of ODEs is the same hard problem as to find a general
solution, which as we know from Ricatti equations is not always
possible.

Most integration methods for second order ODEs are related to another pseudogroup action --
point transformations, which do not act transitively on the space of all such equations.
All linear 2nd order ODEs are point equivalent.

S.\,Lie noticed that ODEs linearizable via point transformations
have necessarily cubic nonlinearity in the first derivatives and
described a general test to construct this linearization map
\cite{Lie$_2$}. Later R.\,Liouville found precise conditions for
linearization \cite{Lio}. But it was A.\,Tresse who first wrote the
complete set of differential invariants for general 2nd order ODEs.

The paper \cite{Tr$_2$} is a milestone in the geometric theory of
differential equations, but mostly one result (linearization of
S.Lie-R.Liouville-A.Tresse) from the manuscript is used nowadays.
In this note we would like to revise the Tresse classification in modern terminology
and provide some alternative formulations and proofs. We make relation to the equivalence
problem more precise and also compare this approach with E.Cartan's equivalence method.

This classification can illustrate the finite representation theorem
for differential invariants algebra, also known as Lie-Tresse
theorem. The latter in the ascending degree of generality was proven
in different sources \cite{Lie$_1$,Tr$_1$,Ov,Ku,Ol,KL$_1$}. In
particular, the latter reference contains the full generality
statement, when the pseudogroup acts on a system of differential
equations $\E\subset J^l(\pi)$ (under regularity assumption, see
also \cite{SS}). We refer to it for details and further references
and we also cite \cite{KLV,KL$_2$} as a source of basic notations,
methods and results.

The structure of the paper is the following. In the first section we
provide a short introduction to scalar differential invariants of a
pseudogroup action and recall what the algebra of relative
differential invariants is. In Section 2 we review the results of
Tresse, confirming his formulae with independent computer
calculation. In Section 3 we complete Tresse's paper by describing
the algebra of absolute invariants and proving the equivalence
theorem (in \cite{Tr$_2$} this was formulated via relative
invariants, which makes unnecessary complications with homogeneity,
and only necessity of the criterion was explained). In Section 4 we
discuss the non-generic 2nd order equations, which contain in
particular linearizable ODEs. Section 5 is devoted to discussion of
symmetric ODEs.

Finally in Appendix (written jointly with V.Lychagin) we provide
another approach to the equivalence problem, based on a reduction of
an infinite-dimen\-si\-onal pseudogroup action to a Lie group
action.

\section{Scalar differential invariants}

We refer to the basics of pseudogroup actions to \cite{Ku,KL$_2$},
but recall the relevant theory about differential invariants (see
also \cite{Tr$_1$,Ol}). We'll be concerned with the infinite Lie
pseudogroup $G=\op{Diff}_\text{loc}(\R^2,\R^2)$ with the
corresponding Lie algebras sheaf (LAS)
$\mathfrak{g}=\mathfrak{D}_\text{loc}(\R^2)$ of vector fields.

The action of $G$ has the natural lift to an action on the space
$J^\infty\pi$ for an appropriate\footnote{In this paper
$\pi=M\times\R$ is a trivial 1-dimensional bundle over
$M\simeq\R^3$, so $J^k\pi=J^kM$.} vector bundle $\pi$, provided we
specify a Lie algebras homomorphism
$\mathfrak{g}\to\mathfrak{D}_\text{loc}(J^0\pi)$. Then we can
restrict to the action of formal LAS $J^\infty(\R^2,\R^2)$.

A function $I\in C^\infty(J^\infty\pi)$ (this means that $I$ is a
function on a finite jet space $J^k\pi$ for some $k>1$) is called a
(scalar absolute) differential invariant if it is constant
along the orbits of the lift of the action of $G$ to $J^k\pi$.

For connected groups $G$ we have an
equivalent formulation: $I$ is an (absolute) differential invariant
if the Lie derivative vanishes $L_{\hat X}(I)=0$ for
all vector fields $X$ from the lifted action of the Lie algebra
$\mathfrak{g}=\op{Lie}(G)$.

Note that often functions $I$ are defined only locally near families
of orbits. Alternatively we should allow $I$ to have meromorphic
behavior over smooth functions (but we'll be writing though about
local functions in what follows, which is a kind of micro-locality,
i.e. locality in finite jet-spaces).

The space $\mathcal{I}=\{I\}$ forms an algebra with respect to usual
algebraic operations of linear combinations over $\R$ and
multiplication and also the composition $I_1,\dots,I_s\mapsto
I=F(I_1,\dots,I_s)$ for any $F\in C^\infty_\text{loc}(\R^s,\R)$,
$s=1,2,\dots$ any finite number. However even with these operations
the algebra $\mathcal{I}$ is usually not locally finitely generated. Indeed,
the subalgebras $\mathcal{I}_k\subset\mathcal{I}$ of order $k$
differential invariants are finitely generated on non-singular
strata with respect to the above operations, but their injective
limit $\mathcal{I}$ is not.

However finite-dimensionality is restored if we add invariant
derivatives, i.e. $\Cc$-vector fields
$\vartheta\in C^\infty(J^\infty\pi)\ot_{C^\infty(M)}\mathfrak{D}(M)$
commuting with the $G$-action on the bundle $\pi$.
These operators map differential invariants to differential invariants
$\vartheta:\mathcal{I}_k\to\mathcal{I}_{k+1}$.

Lie-Tresse theorem claims that the algebra of differential
invariants $\mathcal{I}$ is finitely generated with respect to
algebraic-functional operations and invariant derivatives.

A helpful tool on the practical way to calculate algebra
$\mathcal{I}$ of invariants are relative invariants, because they
often occur on the lower jet-level than absolute invariants. A
function $F\in C^\infty(J^\infty\pi)$ is called a relative scalar
differential invariant if the action of pseudogroup $G$ writes
 $$
g^*F=\mu(g)\cdot F
 $$
for a certain weight, which is a smooth function $\mu:G\to
C^\infty(J^\infty\pi)$, satisfying the axioms of multiplier
representation
 $$
\mu(g\cdot h)=h^*\mu(g)\cdot\mu(h),\quad \mu(e)=1.
 $$

The corresponding infinitesimal analog for an action of LAS $\mathfrak{g}$
is given via a smooth map (the multiplier representation is denoted by the same letter)
$\mu: \mathfrak{g}\to \mathfrak{D}(J^\infty\pi)$,
which satisfies the relations
 $$
\mu_{[X,Y]}=L_{\hat X}(\mu_Y)-L_{\hat Y}(\mu_X),\quad \forall\,
X,Y\in \mathfrak{g},
 $$
Then a relative scalar invariant is a function $F\in C^\infty(J^\infty\pi)$
such that $L_{\hat X}I=\mu_{X}\cdot I$.
In other words (in both cases) the equation $F=0$ is invariant under the action.

Let $\mathfrak{M}=\{\mu_X\}$ be the space of admissible weights\footnote{It is
given via a certain cohomology theory, which will be considered elsewhere.}.
Denote by $\mathcal{R}^\mu$ the space of scalar relative differential invariants
of weight $\mu$. Then
 $$
\mathcal{R}=\bigcup_{\mu\in\mathfrak{M}}\mathcal{R}^\mu
 $$
is a $\mathfrak{M}$-graded module over the algebra of absolute scalar differential invariants
$\mathcal{I}=\mathcal{R}^0$ corresponding to the weight $\mu=0$ for the LAS action
($\mu=1$ for the pseudogroup action).

The space $\mathfrak{M}$ of weights (multipliers) is always a group, but we can transform it into
a $\mathbf{k}$-vector space ($\mathbf{k}=\mathbb{Q}$, $\R$ or $\mathbb{C}$)
by taking tensor product $\mathfrak{M}\otimes\mathbf{k}$ and
considering (sometimes formal) combinations $(I_1)^{\a_1}\cdots(I_s)^{\a_s}$. Then we have:
 $$
\mathcal{R}^\mu\cdot \mathcal{R}^{\bar\mu}\subset \mathcal{R}^{\mu+\bar\mu},\quad
(\mathcal{R}^\mu)^\a\subset\mathcal{R}^{\a\cdot\mu}.
 $$


\section{Tresse classification revisited}

We start by re-phrasing the main results of Tresse
classification\footnote{We use different notations $p$ instead of $z$,
$u$ instead of $\omega$ etc, but this is not crucial.}.

\subsection{Relative differential invariants of 2nd order
ODEs}\label{S21}

The point transformation LAS $\mathfrak{D}_\text{loc}(J^0\R)$, with $J^0\R(x)=\R^2(x,y)$,
equals $\mathfrak{g}=\{\xi_0=a\p_x+b\p_y:a=a(x,y),b=b(x,y)\}$ and it
prolongs to the subalgebra
 \begin{multline*}
\mathfrak{g}_2=\{\xi=a\p_x+b\p_y+A\p_p+B\p_u\}\subset \mathfrak{D}_\text{loc}(J^2\R),\quad
J^2\R=\R^4(x,y,p,u),\\
A=b_x-(a_x-b_y)p-a_yp^2,\ B=B_0+uB_1,\\
B_0=b_{xx}-(a_x-2b_y)_xp-(2a_x-b_y)_yp^2-a_{yy}p^3,\
B_1=-(2a_x-b_y)-3a_yp
 \end{multline*}
where we denote $p=y'$, $u=y''$ the jet coordinates.

Using the notations $D_x=\p_x+p\,\p_y$, $\vp=(dy-p\,dx)(a\p_x+b\p_y)=b-p\,a$ (we'll see soon
these show up naturally), these expressions can be rewritten as
 $$
A=D_x(\vp),\ B_0=D_x^2(\vp),\ B_1=\p_y(\vp)-2D_x(a)
 $$

Thus the LAS $\mathfrak{h}=\mathfrak{g}_2\subset\mathfrak{D}_\text{loc}(J^0\R^3(x,y,p))$
being given we represent a second order ODE as a surface $u=f(x,y,p)$ in
$J^0\R^3(x,y,p)=\R^4(x,y,p,u)$ and $k^\text{th}$ order differential invariants of this ODE are
invariant functions $I\in C^\infty_\text{loc}(J^k\R^3)$ of the prolongation
 \begin{multline*}
\mathfrak{h}_k=\{\hat\xi=a\,\D_x+b\,\D_y+A\,\D_p+\sum_{|\z|\le k}\D_\z^{(k)}(f)\,\p_{u_\z}\}\subset
\mathfrak{D}(J^k\R^3),\\
f=B_0+B_1u-a\,u_x-b\,u_y-A\,u_p\ :\qquad\hat\xi(I)=0.
 \end{multline*}
Here $\D_\z^{(k)}=\D_\z|_{J^k}$ with $\D_\z=\D_x^l\D_y^m\D_p^n$ for
$\z=(l\cdot 1_x+m\cdot 1_y+n\cdot 1_p)$, so that
 \begin{multline*}
\D_\z(f)=\D_\z(B_0)+\sum\frac{|\t|!}{\t!}\Bigl(\D_\t(B_1)u_{\z-\t}
-D_\t(a)u_{\z-\t+1_x}\\
-\D_\t(b)u_{\z-\t+1_y}-\D_\t(A)u_{\z-\t+1_p}\Bigr).
 \end{multline*}
In the above formula we used the usual partial derivatives $\p_x$ etc in the total derivative
operators $\D_\z$ etc. All these operators commute.

It is more convenient, following Tresse, to use the operator $D_x=\p_x+p\,\p_y$
on the base instead and to form the corresponding total derivative
$\hat\D_x=\D_x+p\D_y$. These operators will no longer commute and we need a better notation
for the corresponding non-holonomic partial derivatives.


Denote $u^k_{lm}=\hat\D_x^l\D_y^m\D_p^k(u)$, which equals $u_{lmk}$ mod
(lower order terms).
The first relative invariants calculated by Tresse have order 4 and are:
 \begin{multline*}
I=u^4,\quad H=\\
\ \ u_{20}^2-4u_{11}^1+6u_{02}+u(2u_{10}^3-3u_{01}^2)-u^1(u_{10}^2-4u_{01}^1)
+u^3u_{10}-3u^2u_{01}+u\cdot u\cdot u^4.
 \end{multline*}
In this case the weights form two-dimensional lattice and the relative invariants are
 $$
\mathcal{R}^{r,s}=\{\psi\in C^\infty(J^\infty\R^3):\hat\xi(\psi)=-(rD_x(a)+s\p_y(\vp))\psi\}
 $$
Note that $\hat\xi(\psi)=-(w\,C^w_\xi+q\,C^q_\xi)\psi$ for $w=r$,
$q=s-r$ (weight and quality in Tresse terminology). Here the
coefficients can be expressed as operators of $\xi_0=a\p_x+b\p_y$
and $\xi_1=a\p_x+b\p_y+A\p_p$:
 $$
C^w_\xi=a_x+b_y=\op{div}_{\oo_0}(\xi_0) \text{ and }
C^q_\xi=\p_y(\vp)= \tfrac12\op{div}_{\Omega_0}(\xi_1)
 $$
with $\oo_0=dx\we dy$ the volume form on $J^0\R$ and $\Omega_0=-\oo\we d\oo$ on $J^1\R$, where
$\oo=dy-p\,dx$ is the standard contact form of $J^1\R$. These two
form the base of all weights\footnote{This is a result from a joint discussion with V.Lychagin.
It is important since in Tresse \cite{Tr$_2$} this is an ad-hoc result, based on
the straightforward calculations, but not fully justified.
More details will appear in a separate publication.\label{rkBKVL}}.

There are relative invariant differentiations\footnote{Note that they are differential operators
of the 1st order, obtained from the base derivations via an invariant connection.}
(differential parameters in the classical language):
 \begin{gather*}
\Delta_p=\D_p+(r-s)\frac{u^5}{5u^4}: \mathcal{R}^{r,s}\to\mathcal{R}^{r-1,s+1},\\
\Delta_x=\hat\D_x+u\,\Delta_p+\Bigl((3r+2s)\Bigl(u^1+\frac{3u\,u^5}{5u^4}\Bigr)+(2r+s)\frac{u^4_{10}}{u^4}\Bigr)
: \mathcal{R}^{r,s}\to\mathcal{R}^{r+1,s},\\
\Delta_y=\D_y+\frac{u^5}{5u^4}\,\Delta_x+\Bigl(2\,u^1+\frac{u_{10}^4+u\,u^5}{u^4}\Bigr)\Delta_p+
\Bigl((r+2s)\frac{u_{01}^4}{4u^4}+\qquad\\
\qquad+(3r+2s)\Bigl(\frac{u^2}8+\frac3{20}\frac{u^5(u_{10}^4+u\,u^5+2u^1u^4)}{u^4u^4}\Bigr)
\Bigr): \mathcal{R}^{r,s}\to\mathcal{R}^{r,s+1}.
 \end{gather*}

 \begin{theorem}{\rm\cite{Tr$_2$}}
The space of relative differential invariants $\mathcal{R}$ is generated by
the invariant $H$ and differentiations $\Delta_x,\Delta_y,\Delta_p$ on the generic stratum.
 \end{theorem}

Notice that the latter two 1st order $\Cc$-differential operators have the form:
 \begin{gather*}
\Delta_x=\D_x+p\,\D_y+u\,\D_p+r\Bigl(3u^1+2\frac{u\,u^5+u^4_{10}}{u^4}\Bigr)+
s\Bigl(2u^1+\frac{u\,u^5+u^4_{10}}{u^4}\Bigr),\\
\Delta_y=\frac{u^5}{5u^4}\,\D_x
+\Bigl(1+p\frac{u^5}{5u^4}\Bigr)\,\D_y+\Bigl(2\,u^1+\frac{5u_{10}^4+6u\,u^5}{u^4}\Bigr)\D_p+
r\Bigl(\frac{3u^2}8+\frac{u_{01}^4}{4u^4}\qquad\\
+\frac{19u^1u^5}{10u^4}
+\frac{21(u\,u^5+u_{01}^4)u^5}{20\,u^4\cdot u^4}\Bigr)+
s\Bigl(\frac{u^2}4+\frac{u_{01}^4}{2u^4}+\frac{3u^1u^5}{5u^4}
+\frac{3(u\,u^5+u_{01}^4)u^5}{10\,u^4\cdot u^4}\Bigr),
 \end{gather*}
and so $\Delta_x,\Delta_y,\Delta_p$ are linearly independent everywhere outside $I=0$.

\subsection{Specifications}\label{S22}

Several remarks are noteworthy in relation with the theorem:

{\bf 1.} The number of basic relative differential invariants of pure order $k$ is given
in the following table
 $$
\begin{array}{ccccccccccccc}
k: & 0 & 1 & 2 & 3 & 4 & 5 & 6 & 7 & 8& \dots & k & \dots\\
\#: & 0 & 0 & 0 & 0 & 2 & 3 & 11 & 17 & 24 & \dots & \frac12(k^2-k-8)
\end{array}
 $$
The generators in order 4 are $I\in\mathcal{R}^{-2,3}$ and $H\in\mathcal{R}^{2,1}$; in order 5
$H_{10}=\Delta_x(H)\in\mathcal{R}^{3,1}$, $H_{01}=\Delta_y(H)\in\mathcal{R}^{2,2}$ and
$K=\Delta_p(H)\in\mathcal{R}^{1,2}$; in order 6
are\footnote{We let $H_{ij}=\Delta^i_x\Delta^j_yH$ and
$K_{ij}=\Delta^i_x\Delta^j_yK$, though in \cite{Tr$_2$} there is a difference between
$\Delta_xK$ and $K_{10}$, $\Delta_yK$ and $K_{01}$. Since this only involves a linear
transformation, this is possible.}
$(H_{20},H_{11},H_{02})\in\mathcal{R}^{4,1}\oplus\mathcal{R}^{3,2}\oplus\mathcal{R}^{2,3}$,
$(K_{10},K_{01})\in\mathcal{R}^{2,2}\oplus\mathcal{R}^{1,3}$ and
$\Omega_{ij}^l=u_{ij}^l+\text{(lower terms for certain order on monomials)}\in\mathcal{R}^{i+2-l,j+l-1}$,
$\op{deg}\Omega_{ij}^l=i+j+l=6$, $l>3$:
 $$
\begin{array}{cccccccccccc}
\text{order }k & \text{basic relative differential invariants}\\
4 & I,\ H\\
5 & H_{10},\ H_{01},\ K\\
6 & H_{20},\ H_{11},\ H_{02},\ K_{10},\ K_{01},\ \Omega_{20}^4,\ \Omega_{11}^4,\ \Omega_{02}^4,\
\Omega_{10}^5,\ \Omega_{01}^5,\ \Omega^6
\end{array}
 $$
Thus in ascending order $k$, we must add the generators $I,H$ and then $\Omega_{ij}^{6-i-j}$, $i+j\le2$
(one encounters the relations $\Delta_x(I)=\Delta_y(I)=\Delta_p(I)=0$).
Invariants of order $k>6$ are obtained via invariant derivations from the lower order.

{\bf 2.} The theorem as formulated gives only generators. The relations (differential syzygies)
are the following (also contained in \cite{Tr$_2$}):
 \begin{gather*}
[\Delta_p,\Delta_x]=\Delta_y+\frac{3(3r+2s)}5\frac{\Omega^5_{10}}I\\
[\Delta_p,\Delta_y]=\frac{\Omega^6}{5I}\Delta_x+\frac{\Omega^5_{10}}{I}\Delta_p
-\frac{3(3r+2s)}{20}\frac{\Omega^5_{01}}I\\
[\Delta_x,\Delta_y]=\frac{\Omega^5_{10}}{5I}\Delta_x+\frac{\Omega^4_{20}}{I}\Delta_p
-\frac{3(3r+2s)}4\frac{\Omega^4_{11}}I
 \end{gather*}
together with the following relations for coefficients-invariants
(the first of which is just the application of the above commutator relation)
 \begin{gather*}
\Omega^5_{10}=\frac{5I}{24H}([\Delta_p,\Delta_x]H-\Delta_yH),\qquad
\Omega^5_{01}=\frac49(\Delta_p\Omega^5_{10}-\Delta_x\Omega^6),\\
\Omega^4_{20}=\Delta_p^2H-\frac{\Omega^6}{5I}H,\qquad
\Omega^4_{11}=\frac43(\Delta_p\Omega^4_{20}-\Delta_x\Omega^5_{10}).
 \end{gather*}
It is important that the relation for the last additional invariant of order 6
 $$
\Omega^4_{02}=\frac45(\Delta_y\Omega^5_{10}-\Delta_x\Omega^5_{01}
+\frac{5\Omega^4_{20}\Omega^6+\Omega^5_{10}\Omega^5_{01}}{5I})
 $$
can be considered as definition, while first additional
invariant\footnote{This invariant is important with another approach, see Appendix.}
of order 6
 $$
\Omega^6=u^6-\frac65 \frac{u^5\cdot u^5}{u^4}
 $$
can be obtained from a higher relation via application of the relation
for $[\Delta_p,\Delta_y]$ to $H$ and $K$.

Thus we see that involving syzygy of higher order invariants
(prolongation-projection) we can restore the invariants $I,\O^k_{ij}$ from
$H$ and invariant differentiations $\Delta_j$, as the theorem claims.

{\bf 3.} The theorem specifies the relative invariants only on the
generic stratum. If we take the minimal number of generators
$(H,\Delta_x,\Delta_y,\Delta_p)$, then this stratum is specified by
a number of non-degeneracy conditions of high order.

However if we take more generators $(I,H,\Omega^6,\Delta_x,\Delta_y,\Delta_p)$,
or the collection of basic invariants
$(I,H,\Omega^6,\Omega^5_{10},\dots,\Omega^4_{02},\Delta_x,\Delta_y,\Delta_p)$
 for the completeness in ascending order $k$,
then this condition is very easy: just $I\ne0$.

Notice that the condition $I=0$ is important, since it describes the
singular stratum (see however \S\ref{S42} where this case is
handled).


\section{Classification of 2nd order ODEs}

While a complete classification of relative differential invariants
for 2nd order scalar ODEs was achieved by Tresse, absolute
invariants are not described in \cite{Tr$_2$}. They however can be
easily deduced.

\subsection{Dimensional count}\label{S31}

Let us at first count the number of absolute invariants on a generic
stratum\footnote{This count is independent of Tresse argumentation,
and so together with footnote$^\text{\ref{rkBKVL}}$ it provides a
rigorous proof of the table in \S\ref{S22}.}. This number equals the
codimension of a generic orbit in the corresponding jet-space.

Denote by $\mathcal{O}_k$ the orbit through a generic point in $J^k\R^3(x,y,p)$ of the
pseudogroup of point transformations. Tangent to it is determined by the corresponding LAS
and so we can calculate codimension of the orbit. Indeed, denoting by $\op{St}_k$ the stabilizer
of the LAS $\mathfrak{h}_k$ at the origin we get
 $$
\op{dim}\mathcal{O}_k= \op{codim}\op{St}_k.
 $$
To calculate the stabilizer we should adjust the normal form of the
equation at the origin via a point transformation. This can be done
via a projective configuration (Desargues-type) theorem of \cite{A}
(\S1.6): any 2nd order ODE $y''=u(x,y,p)$, $p=y'$, can be
transformed near a given point to
 $$
y''=\alpha(x)y^2+o(|y|^3+|p|^3).
 $$
Denote by $\mathfrak{m}$ the maximal ideal at the given point
(so $\mathfrak{m}^k$ is the space of functions vanishing to order k).
Then we can suppose that at a given point
 $$
u,u_x,u_y,u_p,u_{xx},u_{xy},u_{xp},u_{yp},u_{pp}\in\mathfrak{m}.
 $$
Therefore the stabilizer $\op{St}_k$ is given by the union of the following conditions
on the coefficients of $\hat\xi\in\mathfrak{h}_k$
(equivalently on coefficients of $\xi_0\in\mathfrak{g}$)
 \begin{gather*}
a\in\mathfrak{m}^{k-2},\ a_{yy}\in\mathfrak{m}^{k-3},\
b\in\mathfrak{m}^{k-1},\ b_{xx}\in\mathfrak{m}^k,\\
a_x\in\mathfrak{m}^{k-2},\ (2a_x-b_y)_y\in\mathfrak{m}^{k-2},\
(a_x-2b_y)_x\in\mathfrak{m}^{k-1}.
 \end{gather*}
Thus the Taylor expansion of $a=a(x,y)$ can contain only the following monomials
 $$
\{x^iy^j:i+j\le k-1\},\ \{x^iy^{k-i}: i>1\},\ \{x^iy^{k+1-i}: i>2\}
 $$
and the allowed monomials for $b=b(x,y)$ are
 $$
\{x^iy^j:i+j\le k\},\ \{x^iy^{k+1-i}: i\ge 1\},\ \{x^iy^{k+2-i}: i\ge 2\}.
 $$
This yields that $\op{codim}(\op{St}_k)$ equals:
 $$
\op{dim}\bigl(\mathbb{C}[x,y]^2/\op{St}_k\bigr)=
\frac{k(k+1)}2+2(k-1)+\frac{(k+1)(k+2)}2+2(k+1)=k^2+6k+1
 $$
and so the number $\imath\!\imath_k$ of the basic differential invariants of order $\le k$ is equal to
 \begin{multline*}
\imath\!\imath_k=\op{codim}\mathcal{O}_k=\op{dim}J^k\R^3-\op{dim}\mathcal{O}_k\\
=3+\frac{(k+1)(k+2)(k+3)}6-(k^2+6k+1)=\frac{k^3-25k+18}6.
 \end{multline*}

As this formula indicates for $k\le4$ the generic orbit is open, so that such stratum
has no absolute invariants (however for $k=4$ there are singular orbits,
so that the relative invariants $I,H$ appear).

In order $k=5$ the formula yields $\imath\!\imath_5=3$ differential invariants.
For $k>5$ we deduce the number of pure order $k$ basic differential invariants:
 $$
\imath\!\imath_k-\imath\!\imath_{k-1}=\frac{k(k-1)}2-4.
 $$

\subsection{Absolute differential invariants}\label{S32}

There are two ways of adjusting a basis on the lattice
$\mathfrak{M}$ of weights via relative invariants. As follows from
specification for $\Z^2$-lattice of weights from \S\ref{S22}, the
basic invariants are
 $$
J_1=I^{-1/8}H^{3/8}\in\mathcal{R}^{1,0},\quad J_2=I^{1/4}H^{1/4}\in\mathcal{R}^{0,1}.
 $$
Another choice, which allow to avoid branching but increase the order, is
 $$
\ti J_1=\frac{H_{10}}H\in\mathcal{R}^{1,0},\quad \ti J_2=\frac{H_{01}}H\in\mathcal{R}^{0,1}.
 $$
Then (choosing $J_i$ or $\ti J_i$) we get isomorphism for $k>4$:
 $$
\mathcal{R}^{r,s}_k/\mathcal{R}^{r,s}_{k-1}\simeq \mathcal{I}_k/\mathcal{I}_{k-1},\qquad
F\mapsto F/(J_1^rJ_2^s).
 $$

Thus with any choice the list of basic differential invariants in order 5 is
 $$
\bar H_{10}=H_{10}/(J_1^3J_2),\ \bar H_{01}=H_{01}/(J_1^2J_2^2),\
\bar K=K/(J_1J_2^2)
 $$
and in pure order 6 is
 \begin{gather*}
\bar H_{20}=H_{20}/(J_1^4J_2),\ \bar H_{11}=H_{11}/(J_1^3J_2^2),\
\bar H_{02}=H_{02}/(J_1^2J_2^3),\ \bar K_{10}=K_{10}/(J_1^2J_2^2),\\
\bar K_{01}=K_{01}/(J_1J_2^3),\ \bar\O^4_{20}=\O^4_{20}/(J_2^3),\ \bar\O^4_{11}=\O^4_{11}/(J_1^{-1}J_2^4),\ \bar\O^4_{02}=\O^4_{02}/(J_1^{-2}J_2^5),\\
\bar\O^5_{10}=\O^5_{10}/(J_1^{-2}J_2^4),\ \bar\O^5_{01}=\O^5_{01}/(J_1^{-3}J_2^5),\
\bar\O^6=\O^6/(J_1^{-4}J_2^5).
 \end{gather*}
Higher order differential invariants can be obtained in a similar way from the
basic relative invariants, but alternatively we can adjust invariant derivations
by letting $\nabla_j=J_1^{\rho_j}J_2^{\z_j}\cdot\Delta_j|_{r=s=0}$
with a proper choice of the weights $\rho_j,\z_j$. Namely we let
 \begin{gather*}
\nabla_p=\frac{J_1}{J_2}\D_p,\qquad \nabla_x=\frac1{J_1}\bigl(\hat\D_x+u\D_p\bigr),\\
\nabla_y=\frac1{J_2}\Bigl(\D_y+\frac{u^5}{5u^4}\hat\D_x+
\bigl(\frac{u_{10}^4}{u^4}+\frac{6u\,u^5}{5u^4}+2u^1\bigr)\D_p\Bigr).
 \end{gather*}
These form a basis of invariant derivatives over $\mathcal{I}$ and we have:
 \begin{gather*}
[\nabla_p,\nabla_x]=-\tfrac18\bar H_{10}\nabla_p-\tfrac38\bar K\nabla_x+\nabla_y,\\
[\nabla_p,\nabla_y]=(\bar\O^5_{10}-\tfrac18\bar H_{01})\nabla_p
+\tfrac15\bar\O^6\nabla_x-\tfrac14\bar K\nabla_y,\\
[\nabla_x,\nabla_y]=\bar\O^4_{20}\nabla_p+(\tfrac15\bar\O^5_{10}+\tfrac38\bar H_{01})\nabla_x
-\tfrac14\bar H_{10}\nabla_y.
 \end{gather*}
The derivations and coefficients can be also expressed in terms of
non-branching invariants
$\ti J_1=\frac83\nabla_xJ_1$ and $\ti J_2=4\nabla_yJ_2$.

 \begin{theorem}
The space $\mathcal{I}$ of differential invariants is generated by the
invariant derivations $\nabla_x,\nabla_y,\nabla_p$ on the
generic stratum.
 \end{theorem}

Indeed, we mean here that taking coefficients of the commutators, adding their derivatives
etc leads to a complete list of basic differential invariants.

On the other hand, if we want to list generators according to the order,
so that invariant derivations only add new in the
corresponding order, then we shall restrict to
$\bar H_{10},\bar H_{01},\bar K$ in order 5, add $\bar\O_{ij}^{6-i-j}$ in order 6 and the
rest in every order is generated from these by invariant derivations with $\nabla_j$.
The relations can be deduced from these of \S\ref{S22}.

\subsection{Equivalence problem}

2nd order ODEs $\E$ can be considered as sections $\mathfrak{s}_\E$ of the bundle $\pi$,
whence we can restrict any differential invariant $J\in \mathcal{I}_k$ to the equation
via pull-back of the prolongation:
 $$
J^\E:=(\mathfrak{s}_\E^{(k)})^*(J)\in C^\infty_\text{loc}(\R^3(x,y,p)).
 $$

Consider most non-degenerate 2nd order ODEs $\E$,
such that\footnote{Here and in what follows one can
assume (higher micro-)local treatment.} $\bar H_{10}^\E$,
$\bar H_{01}^\E$, $\bar K^\E$ are local coordinates on $\R^3(x,y,p)$. Then
the other differential invariants on the equation can be expressed as functions of these:
 $$
\bar H_{ij}^\E=\Phi^\E_{ij}(\bar H_{10}^\E,\bar H_{01}^\E,\bar K^\E),\
\bar K_{ij}^\E=\Psi^\E_{ij}(\bar H_{10}^\E,\bar H_{01}^\E,\bar K^\E),\
\bar \O_{ij}^{k\,\E}=\Upsilon^{k\,\E}_{ij}(\bar H_{10}^\E,\bar H_{01}^\E,\bar K^\E).\!\!
 $$
Due to the relations above we can restrict to the following collection of functions:
 \begin{equation}\label{bascoll}
\Phi^\E_{20}, \Phi^\E_{11}, \Phi^\E_{02},
\Psi^\E_{10}, \Psi^\E_{01},
\Upsilon^{6\,\E}, \Upsilon^{5\,\E}_{10}, \Upsilon^{4\,\E}_{20},
 \end{equation}
the others being expressed through the given ones via the operators of derivations
(which naturally restrict to $\E$ as directional derivatives).

 \begin{theorem}\label{thm3}
Two generic 2nd order differential equations $\E_1$, $\E_2$ are point equivalent
iff the collections
(\ref{bascoll}) of functions on $\R^3$ coincide.
 \end{theorem}

 \begin{proof}
Necessity of the claim is obvious. Sufficiency is based on
investigation of solvability of the corresponding Lie
equation\footnote{It is important not to mix solvability, i.e.
existence of local solutions, with compatibility, i.e. existence of
solutions with all admissible Cauchy data. The latter may be cut by
the compatibility conditions. This confusion occurred in the proof
of Theorem 8.3 from \cite{Y}:
 the Lie equation is not formally integrable except for maximally symmetric case. }
 \begin{equation}\label{LieEq}
\mathfrak{Lie}(\E_1,\E_2)=\{[\vp]^2_z\in J^2(\R^2,\R^2):\,
\vp^{(2)}(\E_1\cap\pi_{2,0}^{-1}(z))=\E_2\cap\pi_{2,0}^{-1}(\vp(z))\},
 \end{equation}
which has finite type. Notice that the prolongation
$\mathfrak{Lie}(\E_1,\E_2)^{(k)}$ consists of the jets
$[\vp]^{k+2}_z$ such that $\vp^{(2)}$ transforms $k$-jets of the
equation $\E_1$ to the $k$-jets of the equation $\E_2$ along the
whole fiber over $z\in J^0\R=\R^2(x,y)$.

 \begin{prop}
Suppose that the system $\mathfrak{Lie}(\E_1,\E_2)$ is formally solvable;
more precisely let $\mathcal{T}\subset \mathfrak{Lie}(\E_1,\E_2)^{(10)}\subset
J^{12}(\R^2,\R^2)$ be such a manifold that $\pi_{12}|_\mathcal{T}$ is
a submersion onto $\R^2$.
Then this system is locally solvable\footnote{A regularity assumption
is needed for this, which is given by the non-degeneracy condition
$d\bar H_{10}^\E\we d\bar H_{01}^\E\we d\bar K^\E\ne0$.},
so that the equations are point equivalent, i.e.
$\exists\vp\in\op{Diff}_\text{\rm loc}(\R^2,\R^2)$:
$\forall z\in\R^2\ [\vp]_z^2\in\mathfrak{Lie}(\E_1,\E_2)$.
 \end{prop}

Indeed, the symbol of the system $\mathfrak{Lie}(\E_1,\E_2)$
(provided it is non-empty, which is usually the case for generic $\E_1,\E_2$)
is the same as for
the symmetry algebra $\op{sym}(\E)$, namely: $g_0=T=\R^2$, $g_1=T^*\ot T$,
$g_2\subset S^2T^*\ot T$ has codimension 4
and no (complex) characteristic covectors, so that $g_3=g_2^{(1)}=0$,
whence $\oplus\,g_i\simeq\op{sym}(y''=0)\simeq\op{sl}_3$.

It should be also noted that the first prolongation
$\mathfrak{Lie}(\E_1,\E_2)^{(1)}\subset J^3(\R^2,\R^2)$ always exists
and is of Frobenius type, while the next one
has proper projection unless the compatibility conditions vanish.

We are interested in solvability of the system,
so we successively add the compatibility conditions. The first belongs
to the space $H^{2,2}(\mathfrak{Lie})\simeq\R^2$,
but it may happen that only one of the components
is non-zero (if both are zero, the system is compatible and we are done, if both
are non-zero we have more equations to add and the process stops earlier).
So we add this equation of the second order to the system of 4 equations and get a
new system $\widetilde{\mathfrak{Lie}}$ of formal $\op{codim}=5$.

Then we continue to add equations-compatibilities
and can do it maximum $\sum\dim g_i=8$ times, so that we
get $3+8=11$-th order condition.
After this we get only discrete set of possibilities for solutions and checking
them we get that either we have a 12-jet solution or there do not exist solutions at all.

In these arguments we adapted dimensional count, i.e. we assumed regularity.
But singularities can bring only zero measure of values (by Sard's lemma), so that
our condition still works even in smooth (not only analytic) situation.

Now let us explain formal solvability for our problem. A jet $[\vp]^{k+2}_z$
belongs to the prolongation $\mathfrak{Lie}(\E_1,\E_2)^{(k)}$ iff
$\vp^{(k+2)}$ transforms $\E_1^{(k)}\cap\pi_{k+2,0}^{-1}(z)$ to
$\E_2\cap\pi_{k+2,0}^{-1}(\vp(z))$. For randomly chosen equations the system
$\mathfrak{Lie}(\E_1,\E_2)$ will be empty over any point $z\in\R^2$
just because none map can transform the whole fiber $\E_1\cap\pi_{2,0}^{-1}(z_1)$
into another fiber $\E_2\cap\pi_{2,0}^{-1}(z_2)$
(example: ODEs $y''=f(x,y,y')$ with polynomial dependence on $p=y'$ of degrees 3 and 4).

The compatibility for the system $\mathfrak{Lie}(\E_1,\E_2)$ of order $k$
are the conditions that $\vp^*$ transforms the restricted order $k$
differential invariants $J^{\E_2}$ into $J^{\E_1}$. Since this is possible
by our assumption, we get prolongation
$\mathcal{T}\subset\mathfrak{Lie}(\E_1,\E_2)^{(10)}$. Moreover this $\mathcal{T}$
will be a submanifold and no singularity issues arise.
This yields us local point equivalence.
 \end{proof}

 \begin{rk}
If differential invariants $J_1\dots J_3$ are independent on equation $\E$,
then there is another way to define invariant derivatives \cite{Lie$_1$,Ol,KL$_1$},
so called Tresse derivatives, which in local coordinates have the form:
$\hat\p/\hat\p J_i=\sum_j[\D_a(J_b)]^{-1}_{ij}\D_j$. In our case, when we
take $\bar H_{10}^\E,\bar H_{01}^\E,\bar K^\E$ as coordinates on the equation,
they are just $\p/\p \bar H_{10}^\E$, $\p/\p \bar H_{01}^\E$, $\p/\p \bar K^\E$,
when restricted to $\E$.
 \end{rk}

Another generic case is when
we have 3 functional independent invariants
among\footnote{We do not know if this is realizable in other
cases, than that described by Theorem \ref{thm3}.}
 \begin{equation}\label{lst}
\bar H_{10}^\E,\ \bar H_{01}^\E,\ \bar K^\E,\
\bar H_{20}^\E,\ \bar H_{11}^\E,\ \bar H_{02}^\E, \
\bar K_{10}^\E,\ \bar K_{01}^\E,\
\bar\O^{6\,\E},\ \bar\O^{5\,\E}_{10}, \bar\O^{4\,\E}_{20}.
 \end{equation}
In this case we can express the rest of invariants through the given 3 basic,
and the classification is precisely the same as in Theorem \ref{thm3}.

There are other regular classes of 2nd order ODEs (in general, equations
are stratified according to functional ranks):
 \begin{enumerate}
 \item Collection (\ref{lst}) has precisely 2 functionally independent invariants,
 \item Collection (\ref{lst}) has only 1 functionally independent invariant,
 \item Collection (\ref{lst}) consists of constants.
 \end{enumerate}

In cases 1 or 2 we can choose basic invariants (2 or 1 respectively --
note that the space of all differential invariants, not only of
collection (\ref{lst}), will then have functional rank 2 or 1) and express
the rest through them. The functions-relations will be again the only
obstructions to point equivalence.

In the latter case all differential invariants are constant on the equation $\E$,
so for the equivalence these (finite number of) constants should coincide.

 \begin{rk}
Cartan's equivalence method provides a canonical frame (on some
bundle over the original manifold), which yields all differential
invariants but with mixture of orders. Otherwise around, given the
algebra of differential invariants, we can choose $J_1,\dots,J_s$
among them, which are functionally independent on a generic
(prolonged) equation. Then $dJ_1,\dots,dJ_s$ will be a canonical
basis of 1-forms, which can work as a (holonomic) moving frame.
Non-holonomic frames can appear upon dualizing invariant
(non-Tresse) derivatives.
 \end{rk}

Let us finally give another formulation of the equivalence theorem. We can consider
collection (\ref{lst}) as a map $\R^3\simeq\E\to\R^{11}$ by varying the point of
our equation $\E$. Thus we get (in regular case) a submanifold of $\R^{11}$ of
dimension 3, 2, 1 or 0 respectively. This submanifold is an invariant (and the previous
formulation was only a way to describe it as a graph of a vector-function):

 \begin{theorem}\label{thm3+}
Two 2nd order regular differential equations $\E_1$, $\E_2$ are
point equivalent iff the corresponding submanifolds in the space of
differential invariants $\R^{11}$ coincide.
 \end{theorem}


\section{Singular stratum: projective connections}

On the space $J^3\R^3(x,y,p)$ the lifted action of the pseudogroup
$\mathfrak{h}$ is transitive. But its lift to the space of 4-jets is
not longer such: There are singular strata, given by the equations
$I=0,H=0$. Moreover they have a singular substratum $I=H=0$ in
itself, on which the pseudogroup action is transitive, so that any
equation from it is point equivalent to trivial ODE $y''=0$
\cite{Lie$_2$,Lio,Tr$_1$}.

In this subsection we consider the singular stratum
$I=0$\footnote{The other stratum $H=0$ can be treated similarly.
Indeed, though the invariants $I,H$ look quite unlike, they are
proportional to self-dual and anti-self-dual components of the
Fefferman metric \cite{F} and this duality is very helpful
\cite{NS}.

Note however that even though it is difficult to solve the PDE $H=0$
without non-local transformations, some partial solutions can be
found using symmetry methods. For instance, a 3-dimensional family
of solutions is $y''=\vp(p)/x$ with
$\vp'''=\dfrac{\vp''(2\vp-2-\vp')}{\vp(\vp-1)}$.
 \label{fnt14}}.
It corresponds to equations of the form
 \begin{equation}\label{2ndOrCub}
y''=\a_0(x,y)+\a_1(x,y)p+\a_2(x,y)p^2+\a_3(x,y)p^3,\quad p=y'.
 \end{equation}
This class of equations is invariant under point transformations.
Moreover it has very important geometric interpretation, namely
such ODEs correspond to projective connections on 2-dimensional manifolds
\cite{C}. We will indicate 3 different approaches to the equivalence
problem.

\subsection{The original approach of Tresse}\label{S41}

The idea is to investigate the algebra of differential invariants,
following S.Lie's method, and then to solve the equivalence problem
via them. In \cite{Tr$_1$} lifting the
action of point transformation to the space $J^k(2,4)$ (jets of maps
$(x,y)\mapsto(\a_0,\dots,\a_3)$) he counts the number of basic
differential invariants of pure order $k$ to be
 $$
\begin{array}{ccccccccccccc}
 k: & 0 & 1 & 2 & 3 & 4 & 5 & 6 & 7 & 8& \dots & k & \dots\\
\#: & 0 & 0 & 0 & 0 & 6 & 8 & 10 & 12 & 14 & \dots & 2(k-1)
\end{array}
 $$
An independent check of this (with the same method as in
\S\ref{S31}) is given in \cite{Y}.

The action of $\mathfrak{g}$ is transitive on the space of 1st jets
and its lift is transitive on the space of second jets $J^2(2,4)$
outside the singular orbit $L_1=L_2=0$, where
 $$
L_1=-\a_{2xx}+2\a_{1xy}-3\a_{0yy} -3\a_3\a_{0x}+\a_1\a_{2x}-6\a_0\a_{3x}
+3\a_2\a_{0y}-2\a_1\a_{1y}+3\a_0\a_{2y}
 $$
 $$
L_2=-3\a_{3xx}+2\a_{2xy}-\a_{1yy} -3\a_3\a_{1x}+2\a_2\a_{2x}-3\a_1\a_{3x}
+6\a_3\a_{0y}-\a_2\a_{1y}+3\a_0\a_{3y}
 $$
These second order operators\footnote{corresponding to $(3k,-3h)$ in
\cite{Tr$_1$}.} were found by S.Lie \cite{Lie$_2$} who showed that
vanishing $L_1=L_2=0$ characterizes trivial (equivalently:
linearizable) ODEs. R.Liouville \cite{Lio} proved that the tensor
 \begin{equation}\label{LiouT}
L=(L_1dx+L_2dy)\ot(dx\we dy),
 \end{equation}
responsible for this, is an absolute differential invariant.

Further on Tresse claims that all absolute differential invariants
can be expressed via $L_1,L_2$, but \cite{Lio,Tr$_2$} do not contain
these formulae. The problem was handled recently by V.Yumaguzhin
\cite{Y} (the whole set of invariants was presented, though not
fully described).

Namely it was shown that the action of $\mathfrak{g}$ in $J^3(2,4)$
is transitive outside the stratum $F_3=0$, where
 \begin{multline*}
F_3=(L_1)^2\D_y(L_2)-L_1L_2(\D_x(L_2)+\D_y(L_1))+(L_2)^2\D_x(L_1)\\
-(L_1)^3\a_3+(L_1)^2L_2\a_2-L_1(L_2)^2\a_1+(L_2)^2\a_0
 \end{multline*}
is the relative differential invariant from \cite{Lio}. The other
tensor invariants can be expressed through these. The invariant
derivations are\footnote{The first one in the relative form was
known already to Liouville \cite{Lio}:
$$\tilde\nabla_1=L_1\D_y-L_2\D_x+m(\D_x(L_2)-\D_y(L_1)):\mathcal{R}^m\to\mathcal{R}^{m+2},$$
where $\mathcal{R}^m$ is the space of weight $m$ relative
differential invariants corresponding to the cocycle
$C_\xi=\op{div}_{\oo_0}(\xi)$, where $\oo_0=dx\we dy$. He was very
close, but did not write the second one.}
 $$
\nabla_1=\frac{L_2}{(F_3)^{2/5}}\D_x-\frac{L_1}{(F_3)^{2/5}}\D_y,\quad
\nabla_2=\frac{\Psi_2}{(F_3)^{4/5}}\D_x-\frac{\Psi_1}{(F_3)^{2/5}}\D_y,
 $$
where
 \begin{gather*}
\Psi_1=-L_1(L_1)_y+4L_1(L_2)_x-3L_2(L_1)_x-(L_1)^2\a_2+2L_1L_2\a_1-3(L_2)^2\a_0,\\
\Psi_2=3L_1(L_2)_y-4L_2(L_1)_y+L_2(L_2)x-3(L_1)^2\a_3+2L_1L_2\a_2-(L_2)^2\a_1.
 \end{gather*}

Now we can get two differential invariants of order 4 as the coefficients of the commutator
 $$
[\nabla_1,\nabla_2]=I_1\nabla_1+I_2\nabla_2.
 $$
Related invariants are the following: one applies the invariant
derivations $\nabla_i$ (extended to the relative invariants) to
$F_3$ and gets another relative differential invariant of the same
weight (the relation here is almost obvious since
$\nabla_1\we\nabla_2$ is proportional to $F_3$). Thus
$\nabla_1(F_3)/F_3,\nabla_2(F_3)/F_3$ are absolute invariants.

To get four more invariants $I_3,\dots,I_6$ of order 4, consider the
Lie equation, formed similar to (\ref{LieEq}) for the cubic 2nd
order ODEs (\ref{2ndOrCub}), see (\ref{LEC}). After a number of
prolongation-projection we get a Frobenius system, and its
integrability conditions yield the required differential invariants
(in \cite{Y} these are obtained in a different but seemingly
equivalent way).

Now we can state that the algebra $\mathcal{I}$ is generated by the
invariants $I_1,\dots,I_6$ together with the invariant derivatives
$\nabla_1,\nabla_2$. An interesting problem is to describe all
differential syzygies between these generators.

\subsection{The second Tresse approach}\label{S42}

The invariants of \S\ref{S22} are not defined on the stratum $I=0$
due to the fact that most expressions contain $I$ in denominator.
But due to footnote$^\text{\ref{fnt14}}$ the relative invariants $I,H$
are on equal footing. And in fact Tresse in \cite{Tr$_2$} constructs
another basis of relative invariants with $H$ in denominator.

Thus if we restrict this set to the stratum $I=0$ minus the trivial equation,
corresponding to $I=H=0$, we get relative/absolute differential invariants
of the ODEs (\ref{2ndOrCub}). For instance $H$ is proportional to $L_1+L_2p$,
which under substitution of $p=\frac{dy}{dx}$ is proportional to the tensor $L$.
The other invariants are rational functions in $p$ on the cubics (\ref{2ndOrCub}),
which may be taken in correspondence with the invariants of
the approach from \S\ref{S41}.

The proposed idea can be viewed as a change of coordinates in the algebra
$\mathcal{I}$. Yet, another approach was sketched in \cite{Tr$_1$}, which can be
called a non-local substitution.

Namely by a point transformation Tresse achieves $L_2=0$, and so brings
the tensor $L_1dx+L_2dy$ to the form $\l\,dx$. Then the point transformation
pseudogroup is reduced to the triangular pseudogroup $x\mapsto X(x)$, $y\mapsto Y(x,y)$,
and the invariants are generated by the invariant derivatives $\Delta_x,\Delta_y$ and the invariants $B,C,D$ of orders 1, 2, 2 respectively (\cite{Tr$_1$}, ch.III),
which though do not correspond to the orders in the approach of \S\ref{S41}.

\subsection{Lie equations}

Let $\mathfrak{s}_\E:\R^2\to\R^4$ be the map
$(x,y)\mapsto(a_0,a_1,a_2,a_3)$ corresponding to a 2nd order ODE $\E$
(\ref{2ndOrCub}). With two such ODEs we relate the Lie equation
on the equivalence between them:
 \begin{equation}\label{LEC}
\mathfrak{Lie}(\E_1,\E_2)=\{[\vp]^2_z\in J^2(2,2):
\hat\vp\bigl(\mathfrak{s}_{\E_1}(z)\bigr)=\mathfrak{s}_{\E_2}(\vp(z))\},
 \end{equation}
where $\hat\vp:\R^2\times\R^4\to\R^2\times\R^4$ is the lift
of a map $\vp:\R^2\to\R^2$ to a map
of ODEs (\ref{2ndOrCub}). On infinitesimal level, the lift of a vector field
$X=a\,\p_x+b\,\p_y$ is
 \begin{multline*}
\hat X= a\,\p_x+b\,\p_y+(b_{xx}+\a_0(b_y-2a_x)-\a_1b_x)\,\p_{\a_0}\\
+(2b_{xy}-a_{xx}-3\a_0a_y-\a_1a_x-2\a_2b_x)\p_{\a_1}
+(b_{yy}-2a_{xy}-2\a_1a_y-\a_2b_y-3\a_3b_x)\p_{\a_2}\\
+(-a_{yy}-\a_2a_y+\a_3(a_x-2b_y))\p_{\a_3}.
 \end{multline*}

For one equation $\E_1=\E_2$ infinitesimal version of the finite Lie
equation $\mathfrak{Lie}(\E,\E)$ describes the symmetry algebra
(which more properly should be called a Lie equation \cite{KSp})
$\op{sym}(\E)$: it is formed by the solutions of
 \begin{equation}\label{LECc}
\mathfrak{lie}(\E)=\{[X]^2_z\in J^2(2,2):
\hat X\in T_{\mathfrak{s}_\E(z)}[\mathfrak{s}_\E(\R^2)]\}.
 \end{equation}
The basic differential invariants of the pseudogroup
$\op{Diff}_{\text{loc}}(\R^2,\R^2)$ action on ODEs (\ref{2ndOrCub})
arise as the obstruction to formal integrability of the equation
$\mathfrak{lie}(\E)$ (for the equivalence problem
$\mathfrak{Lie}(\E_1,\E_2)$, but the investigation is similar). In
coordinates, when the section $\mathfrak{s}_\E$ is given by four
equations $\a_i-\a_i(x,y)=0$, overdetermined system (\ref{LEC}) is
written as
 \begin{alignat*}{1}
b_{xx}+\a_0(b_y-2a_x)-\a_1b_x&=a\,\a_{0x}+b\,\a_{0y}\\
2b_{xy}-a_{xx}-3\a_0a_y-\a_1a_x-2\a_2b_x&=a\,\a_{1x}+b\,\a_{1y}\\
b_{yy}-2a_{xy}-2\a_1a_y-\a_2b_y-3\a_3b_x&=a\,\a_{2x}+b\,\a_{2y}\\
-a_{yy}-\a_2a_y+\a_3(a_x-2b_y)&=a\,\a_{3x}+b\,\a_{3y}
 \end{alignat*}

The symbols $g_i\subset S^iT^*\ot T$ are: $g_0=T=\R^2$,
$g_1=T^*\ot T\simeq\R^4$, $g_2\simeq\R^2$ and $g_{3+i}=0$ for $i\ge0$.
The compatibility conditions belong to the Spencer cohomology group
$H^{2,2}(\mathfrak{lie})\simeq\R^2$: this is equivalent to the tensor $L$ of
(\ref{LiouT}). If $L=0$, the equation is integrable\footnote{not only formally,
but also locally smoothly due to the finite type of $\mathfrak{lie}$.} and the solution
space is the Lie algebra $\op{sl}_3$.

If $L\ne0$, the equation $\mathfrak{lie}_0=\mathfrak{lie}(\E)$ has
prolongation-projection\footnote{This means that the Lie equation has
the first prolongation $\mathfrak{lie}^{(1)}\subset J^3(2,2)$, but the
next prolongation exists only over the jets of vector fields $X$, preserving the
tensor $L$.}
$\mathfrak{lie}_1=\pi_{4,1}(\mathfrak{lie}^{(2)})$ with symbols
$g_0=T$, $\bar g_1\simeq \R^2\subset g_1$, $g_2\simeq\R^2$ and
$g_{3+i}=0$ for $i\ge0$.

After prolongation-projection, one gets the equation $\mathfrak{lie}_2$
with symbols $g_0=T$, $\tilde g_1\simeq \R^1\subset\bar g_1$
and $g_{2+i}=0$ for $i\ge0$.
This equation has the following space of compatibility conditions:
$H^{1,2}(\mathfrak{lie}_3)\simeq\R^1$. It yields the condition of
the third order in $\a_i$: $F_3=0$ (this, together with other invariants
\cite{R}, characterizes equations with 3-dimensional symmetry algebra,
namely $\op{sl}_2$).

If $F_3\ne0$, then the prolongation-projection yields the equation $\mathfrak{lie}_3$
with $g_0=T$ and $g_{1+i}=0$ for $i\ge0$. The compatibility conditions are given by
the Frobenius theorem and this provides the basis of differential invariants.

{\bf Remarks.} {\bf 1.} The idea to reformulate equivalence problem
via solvability of an overdetermined system appeared in S.Lie's
linearization theorem, where he showed that an ODE (\ref{2ndOrCub})
is point equivalent to the trivial equation $y''=0$ iff the system
(see \cite{Lie$_2$}, p.365 (we let $z=c,w=C$ etc), and also
\cite{IM})
 \begin{alignat*}{2}
 \frac{\p w}{\p x} &= zw-\a_0\a_3-\frac13\frac{\p\a_1}{\p y}+
\frac23\frac{\p\a_2}{\p x}, &\quad
 \frac{\p z}{\p x} &=
z^2-\a_0w-\a_1z+\frac{\p\a_0}{\p y}+\a_0\a_2,\\
 \frac{\p w}{\p y} &= -w^2+\a_2w+\a_3z+\frac{\p\a_3}{\p x}-\a_1\a_3,
&\quad
 \frac{\p z}{\p y} &= -zw+\a_0\a_3-\frac13\frac{\p\a_2}{\p x}+
\frac23\frac{\p\a_1}{\p y}.
 \end{alignat*}
is compatible. The compatibility conditions here are given by the
Frobenius theorem: $L_1=L_2=0$. In fact, the system can be
transformed into a linear system\footnote{S.Lie considers finite
transformations, whence the non-linearity. A projective
transformation is needed to change this into a linear system, while
the infinitesimal analog --- our Lie equation $\mathfrak{lie}(\E)$
-- is linear from the beginning.}, which is equivalent to half of
our once prolonged Lie equation $\mathfrak{lie}^{(1)}$ (Lie
considers combinations of the unknown functions-component of the
point transformation, that's why in the third order we get only
$4=8/2$ equations, the second half of equations was not much used by
him).

{\bf 2.} Other ways of getting differential invariants arise from problems which
have projectively invariant answers. For instance the following system arose in
3 independent problems:
 $$
u_y=P_0[u,v,w],\ \ u_x+2v_y=P_1[u,v,w],\ \ 2v_x+w_y=P_2[u,v,w],\ \
w_x=P_3[u,v,w],
 $$
where $P_i[u,v,w]$ are linear operators of a special type, with
coefficients being smooth functions in $x,y$. This system can be
obtained similar to $\mathfrak{lie}$ from the condition of existence
of Killing tensors\footnote{Substitution
$u=3\x_y,w=3\eta_x,v=-(\x_x+\eta_y)/2$ transforms this system to the
kind
$\x_{yy}=\dots,2\x_{xy}-\eta_{yy}=\dots,\x_{xx}-2\eta_{xy}=\dots,-\eta_{xx}=\dots$,
which has the same symbol as (\ref{LECc}).}.

In \cite{K} solvability of this system lead to an invariant characterization of
Liouville metrics, in \cite{BMM} to normal forms of metrics with transitive group
of projective transformations and in \cite{BDE} -- to the condition of
local metrisability of projective structures on surfaces.

All these problems have the answers (for instance, in the first mentioned paper,
the number of Killing tensors of a metric), which are projective invariants.
Thus they provide projective differential invariants and in turn can be expressed via
any basis of them.

{\bf 3.} Many papers addressed the higher-dimensional version of the same
equivalence problem (which is surprisingly easier, because the Lie equation
is more overdetermined). In Cartan \cite{C} this is the study of the projective
connection. Refs. \cite{Th,Lev} address the algebra of
scalar projective differential invariants.

However in neither of these approaches the Tresse method was
superseded. For instance, in the latter reference even the number of
differential invariants for the 2-dimensional case was not
determined. On the other hand, the method of Lie equations allows to
obtain the algebra of projective invariants in the
higher-dimensional case as well.


\section{Application to symmetries}


At the end of \cite{Tr$_2$} a classification of symmetric equations
is given. It turns out that the symmetry algebra can be of
dimensions 8, 3, 2, 1 or 0. This follows from the study of
dependencies among differential invariants, and it is not obvious
that this automatically applies to all singular strata (but it is
true).

Thus if $\dim\op{Sym}(\E)=8$ the ODE is equivalent to the trivial
$y''=0$. If $\dim\op{Sym}(\E)=3$, the normal forms are ($y'=p$):
 \begin{alignat*}{2}
y''&=p^a &\qquad y''&=\frac{(cp+\sqrt{1-p^2})(1-p^2)}{x} \\
y''&=e^p &\qquad y''&=\pm(xp-y)^3.
 \end{alignat*}
Only the last form belongs to the singular stratum $I=0$.

Due to symmetry between $I$ and $H$, there should be corresponding
normal form with $H=0$. Here one can be mislead since direct
calculations shows that none has vanishing $H$. The reason is
however that Tresse uses Lie's classification of Lie algebras
representation by vector fields on the plane. For 3-dimensional
algebras Lie used normal forms over $\C$, and ineed the third normal
form has $H=0$ for the parameter $c=\pm i$. Thus over $\R$ the above
normal forms should be extended.

As the symmetry algebra reduces to dimensions $2$ we have the
respective normal forms
 $$
y''=\psi(p)\qquad\text{ and }\qquad y''=\psi(p)/x.
 $$
It is important that for singular strata the classification shall be
finer. This is almost obvious for projective connections (cubic
$\psi$), but for metric projective connections this is already
substantial, see \cite{BMM}.

The case $\dim\op{Sym}(\E)=1$ has only one quite general form
$y''=\psi(x,p)$ with an obvious counterpart for projective
connections (for the metric case see \cite{Ma}).

 \begin{rk}
When the transformation pseudogroup reduces from point to
fiber-preserving (triangular) transformations of
$J^0\R(x)=\R^2(x,y)$, the algebra of differential invariants grows,
but the symmetric cases change completely. In particular, the
symmetry algebra can have dimensions 6, 3, 2, 1 or 0 \cite{KSh,HK}.
 \end{rk}

Not much is known about the criteria for having the prescribed
dimension of the symmetry algebra, except for the corollary of
Lie-Liouville-Tresse theorem: $\dim\op{Sym}(\E)=8$ iff $L=0$.

In \cite{Tr$_2$} the following was claimed (we translate it to the
language of absolute differential invariants):
 \begin{itemize}
 \item[$\diamond$]
The symmetry algebra is 3-dimensional iff all the differential
invariants (on the equation) are constant.
 \item[$\diamond$]
The symmetry algebra is 2-dimensional iff the space of differential
invariants has functional rank 1, i.e. any two of them are
functionally dependent (Jacobian vanishes).
 \item[$\diamond$]
The symmetry algebra is 1-dimensional iff the space of differential
invariants has functional rank 2 (all $3\times3$ Jacobians vanish).
 \end{itemize}
This (unproved in Tresse, but correct statement) is however
inefficient, since checking all invariants is not practically
possible. Here is an improvement:
 \begin{theorem}
The above claims hold true if we restrict to the basic absolute
differential invariants of order $\le6$ described in \S\ref{S22}.
 \end{theorem}
What is the minimal collection of differential invariants answering
the above question is seemingly unknown (except for the case
$\dim\op{Sym}(\E)=3$ for $I=0$ handled in \cite{R}).


\appendix
\section{Appendix: Another approach}
 \begin{center}{\bf B.\,Kruglikov, V.\,Lychagin}
 \end{center}

{\bf 1.} Consider the stabilizer $\mathfrak{l}_8\subset\mathfrak{g}$
of a point $(x,y)\in \R^2(x,y)=J^0\R(x)$. We can choose coordinates
so that $x=y=0$. The vector fields generating this subalgebra of
vector fields of $F_{x,y}=\R^2(p,u)=\pi_{2,0}^{-1}(x,y)$ are
 $$
\mathfrak{l}_8=\langle
 \p_p,\ \p_u,\ p\,\p_p,\ u\,\p_u,\ p\,\p_u,\
 p^2\p_u,\ p^3\p_u,\ p^2\p_p+3p\,u\,\p_u
\rangle
 $$
This is an 8-dimensional Lie algebra with Levi decomposition
$\mathcal{R}_5\ltimes\op{sl}_2$, where $\mathcal{R}_5$ is the
radical, which is a solvable Lie algebra with 4-dimensional
(commutative) nil-radical.

In $F_{x,y}$ the 2nd order equation $\E$ is a
curve\footnote{depending parametrically on $x,y$.} $u=f(p)$. Since
in equivalence problem we can transform one base point to another by
a point transformation (any point to any if the equation possesses a
2-dimensional symmetry group, transitive on the base), the
equivalence problem is reduced to the equivalence of curves on the
plane $\R^2(p,u)$ with respect to the Lie group $\mathfrak{l}_8$
action.

The action lifts to the spaces $J^k\R(p)=\R^{k+2}(p,u,u_p,\dots)$
and is transitive up to 3rd jets. The first singular orbit appears
in the space $J^4\R(p)$ and is $\mathcal{S}_1=\{G_4=u^4=0\}$ (we
continue to use the same notations as above, so that
$u^4=u_{pppp}$). The next singular orbit (different from
prolongation of this one) appears in the space $J^6\R(p)$ and is
$\mathcal{S}_2=\{G_6=5u^4u^6-6u^5\cdot u^5=0\}$.

Notice that the second equation belongs to the prolongation of the
first: $\mathcal{S}_2\subset\mathcal{S}_1^{(2)}$ (so it is a
sub-singular orbit). The functions $G_4,G_6$ are relative
differential invariants. In this case the weights can be chosen via
cocycles $C_\xi^r=\op{div}_{\oo_0}(\xi)-\frac12\op{div}_{\O}(\xi_1)$
and $C_\xi^s=\frac12\op{div}_{\O}(\xi_1)$, where $\oo_0=dp\we du$,
$\O=-\oo\we d\oo=dp\we du\we du_p$ (for $\oo=du-u_p\,dp$) and
$\xi=A\p_p+B\p_u$,
$\xi_1=X_f=A\p_p+B\p_u+(\D_p(B)-\p_p(A)u_p)\p_{u_p}$,
$\hat\xi=X_f^{(\infty)}$ with
 $$
f=B-A\,u_p,\quad A=a_0+a_1p+a_2p^2,\quad
B=b_0+b_1p+b_2p^2+b_3p^3+b_4u+3a_2p\,u.
 $$
Denoting $\mathcal{R}^{r,s}=\{\psi\in
C^\infty(J^\infty\R):\hat\xi(\psi)=-(r\,C_\xi^r+s\,C_\xi^s)\psi\}$,
we get\footnote{Note that since the group is changed the grading is
changed as well. In particular $G_4$, which formally coincides with
$I$ of section \S\ref{S21}, has a different grading.}:
 $$
G_4\in\mathcal{R}^{4,-1},\ G_6\in\mathcal{R}^{10,-2}.
 $$

The relative invariant derivative here equals
 $$
\diamondsuit_p=\D_p-\frac{2r+3s}5\frac{u^5}{u^4}:\mathcal{R}^{r,s}\to\mathcal{R}^{r+1,s}.
 $$
It acts trivially on $G_4$, but from its action on $G_6$ we can
extract an absolute invariant. Indeed since
$G_4/\sqrt{G_6}\in\mathcal{R}^{-1,0}$, we have:
$\diamondsuit_p(G_4/\sqrt{G_6})\in\mathcal{R}^{0,0}=\mathcal{I}$ and
the latter expression is non-zero.

Actually the action of our 8-dimensional group has open orbits through
generic points on $J^6\R(p)$ and the first absolute differential
invariant appear at order 7 and equals\footnote{Superscript after
brackets means the power, while the others are indices.}
 $$
I_7=\frac{25(u^4)^2u^7+84(u^5)^3-105u^4u^5u^6}{(G_6)^{3/2}},
 $$
which coincides with $-10\,\diamondsuit_p(G_4/\sqrt{G_6})$.

In each higher order we get 1 new differential invariant. They
determine Tresse derivative (see \cite{KL$_1$}), but we can obtain
the absolute invariant derivative directly:
 $$
\square_p=\left.\frac{G_4}{\sqrt{G_6}}\diamondsuit_p\right|_{r=s=0}
=\frac{u^4}{\sqrt{5u^4u^6-6u^5\cdot
u^5}}\,\D_p:\mathcal{I}\to\mathcal{I}.
 $$
This can be expressed via invariants of \S\ref{S32} as
$\frac1{\sqrt{5\O^6}}\nabla_p$.

Thus on the generic stratum every differential invariant is
(micro-locally) a function of the invariants
$I_7,\square_p(I_7),\square_p^2(I_7),\dots$ of orders $7,8,9,\dots$

{\bf 2.} It is easy to see that the class of cubic curves $u=Q_3(p)$
is invariant with respect to the Lie group
$\mathfrak{L}_8=\op{Exp}(\mathfrak{l}_8)$ action. This 4-dimensional
space forms a singular orbit, on which $\mathfrak{L}_8$ acts
transitively.

Another singular orbit is $\mathcal{S}_2$, which is a 6-dimensional
manifold of curves $u=(a_0+a_1p+a_2p^2+a_3p^3)+b/(p-c)$, and
$\mathfrak{L}_8$ acts transitively there on the generic stratum. The
singular stratum is given by the equation $b=0$ and coincides with
the previous stratum $\mathcal{S}_1$.

{\bf 3.} Consider the stabilizer
$\mathfrak{l}_6=\op{St}_0\subset\mathfrak{h}=\mathfrak{g}_2$ of a
point $x_2=(x,y,p,u)\in J^0\R^3(x,y,p)=J^2\R(x)$. Since the
pseudogroup $\mathfrak{g}_2$ acts transitively on $J^2\R(x)$, choice
of $x_2$ is not essential, in particular we can take a coordinate
representative $(x,y,0,0)$. This Lie algebra is generated by
(prolongation of) the fields
 $$
\mathfrak{l}_6=\langle
 p\,\p_p,\ u\,\p_u,\ p\,\p_u,\
 p^2\p_u,\ p^3\p_u,\ p^2\p_p+3p\,u\,\p_u
\rangle
 $$
and is solvable (with 4-dimensional nil-radical of length 2). Thus
investigation of its invariants is easier thanks to S.Lie quadrature theorem.

Moreover we can continue and take the stabilizer
$\op{St}_3\subset\mathfrak{g}_5$ of a point in $J^3\R^3(x,y,p)$,
where the action of $\mathfrak{h}$ is still transitive. This
stabilizer is a trivial 1-dimensional extension of the 2D solvable
Lie algebra.

{\bf 4.} With all these approaches we get enough invariants to
pursue classification in generic case (and even to deal with
singular orbits). Indeed, we get a subalgebra $\mathfrak{V}$ in the
algebra of all differential invariants $\mathcal{I}$, which we can
restrict to the equation. Provided that there are invariants in
$\mathfrak{V}$ independent, on our 2nd order ODE $\E$, we solve the
equivalence problem.

In general to restore the whole algebra $\mathcal{I}$ we must add to
this vertical invariants algebra $\mathfrak{V}$ invariant
derivatives $\nabla_x$, $\nabla_y$. This is similar to Liouville
approach for cubic 2nd order ODEs (retrospectively after \cite{C} --
projective connections), who found in \cite{Lio} only a subalgebra
of (relative) differential invariants (the second relevant invariant
derivative
$\nabla_2=\Psi_2\D_x-\Psi_1\D_y+...:\mathcal{R}^r\to\mathcal{R}^{r+4}$
and the corresponding part of differential invariants was not
established).


 \vspace{-10pt} \hspace{-20pt} {\hbox to 12cm{ \hrulefill }}
\vspace{-1pt}

{\footnotesize \hspace{-10pt} Institute of Mathematics and
Statistics, University of Troms\o, Troms\o\ 90-37, Norway.

\hspace{-10pt} E-mails: \quad kruglikov\verb"@"math.uit.no, \quad
lychagin\verb"@"math.uit.no.} \vspace{-1pt}

\end{document}